\documentclass[11pt]{epiarticle}
\usepackage{epimath}
\usepackage{amssymb,latexsym,amsmath} 
\setpapertype{A4}

\usepackage[english]{babel}

\newtheorem{thm}{Theorem}[section]

\newtheorem{defin}{Definition}[section]
\newtheorem{lem}{Lemma}[section]
\newtheorem{exa}{Example}[section]
\newtheorem{rem}{Remark}[section]

\newtheorem{pro}{Proposition}[section]

\title{Dual Ramanujan-Fourier series}
\titlemark{Dual Ramanujan-Fourier series}
\author{Noboru Ushiroya}
\authormark{N.\ Ushiroya}
\date{} 
\journal{}

\begin{document}

\maketitle




\begin{prelims}

\def\abstractname{Abstract}
\abstract{Let $c_q(n)$ be the Ramanujan sums. Many results concerning Ramanujan-Fourier series $f(n)=\sum_{q=1}^\infty a_q c_q (n)$ are obtained by many mathematicians. In this paper we study series of the form $f(q)=\sum_{n=1}^\infty a_n c_q (n)$, which we call dual Ramanujan-Fourier series. We extend Lucht's theorem and Delange's theorem to this case and obtain some results.}




\end{prelims}


\section{Introduction}

For $q,n \in \mathbb{N}=\{ 1,2, \cdots \}$, the Ramanujan sums $c_q (n)$ are defined in \cite{Ramanujan} by
\begin{equation}
c_q(n)=\sum_{\substack{k=1 \\ (k,q)=1}}^q \exp (\frac{2 \pi i kn}{q}) , \nonumber
\end{equation}
where $(k,q)$ is the greatest common divisor of $k$ and $q$.
Let $f:\mathbb{N} \mapsto \mathbb{C} $ be an arithmetic function. Ramanujan \cite{Ramanujan} investigated its Ramanujan-Fourier series which is an infinite series of the form
\begin{equation}
\label{eq1-1}
f(n) = \sum_{q=1}^\infty a_q c_q (n),
\end{equation}
where $a_q$ are called the Ramanujan-Fourier coefficients of $f$, and he obtained the following results.
\begin{align}
& \frac{\sigma_s (n)}{n^s} =\zeta(s+1) \sum_{q=1}^\infty \frac{c_q (n)}{q^{s+1}} , \label{eq:sigma} \\
& \frac{\varphi_s (n)}{n^s} =\frac{1}{\zeta(s+1)} \sum_{q=1}^\infty \frac{\mu (q)}{\varphi_{s+1} (q)} c_q (n) , \label{eq:varphi} \\
& \tau (n)=- \sum_{q=1}^\infty \frac{\log q}{q} c_q (n) , \label{eq:tau} \\
& r(n) =\pi \sum_{q=1}^\infty \frac{(-1)^{q-1}}{2q-1} c_{2q-1} (n) , \label{eq:r}
\end{align}
where $\sigma_s (n)=\sum_{d | n} d^s$ with $s>0$,  \ $\zeta(s)$ is the Riemann zeta function, \ $\varphi_s (n)= n^s \prod_{p | n} (1-1/p^s) $, \ $\tau (n)=\sum_{d | n} 1$ , \ $\mu$ is the M$\rm{\ddot{o}}$bius function and $r(n)$ is the number of representations of $n$ as the sum of two squares.

 Ramanujan \cite{Ramanujan} also investigated dual Ramanujan-Fourier series of the form
\begin{equation}
\nonumber
f(q) = \sum_{n=1}^\infty a_n c_q (n),
\end{equation}
and he obtained the following results.
\begin{align}
 &(\mathrm{id}^{1-s}*\mu)(q)=\varphi_{1-s} (q)=\frac{1}{\zeta(s)} \sum_{n=1}^\infty \frac{c_q(n)}{n^s} \quad \mathrm{if}  \ \ s>1, \label{eq:ra-1}\\
& \Lambda (q)=- \sum_{n=1}^\infty \frac{c_q (n)}{n} \quad \mathrm{if}  \ \ q \geqq 2 , \label{eq:ra-2}
\end{align}
where $\mathrm{id}$ is the function $\mathrm{id} (n)=n$, $f *g$ denotes the Dirichlet convolution of $f$ and $g$, and  $\Lambda (q) $ denotes the von Mangoldt function.

 We investigate dual Ramanujan-Fourier series and obtain theorems which are extensions of the results due to Delange and Lucht. Several examples are given. The method used in this paper is quite elementary.


\section{Preliminaries}

Let $\delta (n)= \begin{cases} 1 \ \ \mathrm{if} \ \ n=1 \\ 0 \ \ \mathrm{if} \ \ n>1 \end{cases}$  and let $\delta (m,n)= \begin{cases} 1 \ \ \mathrm{if} \ \ m=n \\ 0 \ \ \mathrm{if} \ \ m \neq n. \end{cases}$ \\
We set $D(m,n)=m \delta(m,n)$. Obviously, $D(m,n)=D(n,m)$ holds.
 
Let $f$, $g: \mathbb{N} \mapsto \mathbb{C}$ be arithmetic functions.
The Dirichlet convolution of $f$ and $g$ is defined by
\begin{equation}
\nonumber
(f*g)(n)=\sum_{d | n} f(d) g( n/d). 
\end{equation}

For two arithmetic functions, one of which is a function of one variable, the other a function of two variables, we define similar types of convolutions as follows.
\begin{defin}
Let $f: \mathbb{N} \mapsto \mathbb{C}$ be an arithmetic function and $g: \mathbb{N} \times \mathbb{N} \mapsto \mathbb{C}$ be an arithmetic function of two variables.
We define $f \underset{\ell}{*} g : \mathbb{N} \times \mathbb{N} \mapsto \mathbb{C}$ and $g \underset{r}{*} f : \mathbb{N} \times \mathbb{N} \mapsto \mathbb{C} $ as follows.
\begin{align*}
(f \underset{\ell}{*} g)(q,n)=& (f( \cdot) *g( \cdot, n))(q)=\sum_{d | q} f(\frac{q}{d}) g( d,n) , \\
(g \underset{r}{*} f)(q,n)=& (g( q, \cdot) * f( \cdot))(n)=\sum_{d | n}  g( q,d) f(\frac{n}{d}). 
\end{align*}

\end{defin}

It is clear that the following lemma holds.


\begin{lem}
\label{lem2-1}
Let $f,h: \mathbb{N} \mapsto \mathbb{C}$ be arithmetic functions and let $g: \mathbb{N} \times \mathbb{N} \mapsto \mathbb{C}$ be an arithmetic function of two variables.
Then we have
\begin{align*}
& (f \underset{\ell}{*} g) \underset{r}{*} h =f \underset{\ell}{*} (g \underset{r}{*} h)  ,  \\
& f \underset{\ell}{*} (h \underset{\ell}{*} g)= (f*h) \underset{\ell}{*} g ,  \\
& (g \underset{r}{*} f) \underset{r}{*}  h= g \underset{r}{*} (f*h) . \nonumber 
\end{align*}
\end{lem}

We note that $((f \underset{\ell}{*} g) \underset{r}{*} h)(q,n) $ can also be written as $\sum_{d_1 | q, \ d_2 | n} f(q/d_1) g(d_1,d_2) h(n/d_2).$
We simply write $f \underset{\ell}{*} g \underset{r}{*} h$ instead of $(f \underset{\ell}{*} g) \underset{r}{*} h$ or $f \underset{\ell}{*} (g \underset{r}{*} h)$.

It is easy to see that the following lemma holds.


\begin{lem}
\label{lem2-2}
Let $f: \mathbb{N} \mapsto \mathbb{C}$ be an arithmetic function. Then we have
\begin{align*}
(f \underset{\ell}{*} D)(q,n) & =I_{n \mid q}  f(\frac{q}{n}) n =\left\{ \begin{array}{cl} 
  f(\frac{q}{n}) n  & \mathrm{if} \ \ n \mid q  \\ 0 &  \mathrm{if}  \ \ n \nmid q ,  \end{array}  \right.  \\
(D \underset{r}{*} f)(q,n) & =I_{q \mid n} f(\frac{n}{q}) q = \left\{ \begin{array}{cl} 
  f(\frac{n}{q}) q  & \mathrm{if} \ \ q \mid n  \\ 0 &  \mathrm{if}  \ \ q \nmid n , \end{array}  \right.
\end{align*}
where $I_{n \mid q}=\begin{cases} 1 \ \ \mathrm{if} \ \ n \mid q \\ 0 \ \ \mathrm{if} \ \ n \nmid q. \end{cases} $
\end{lem}
\begin{proof}
By definiton, we have
\begin{align*}
(f \underset{\ell}{*} D)(q,n)=\sum_{d | q} f(\frac{q}{d}) D(d,n)=\sum_{d | q} f(\frac{q}{d}) n \delta (d,n)=I_{n \mid q}  f(\frac{q}{n}) n. 
\end{align*}
The proof of the second assertion is similar. \qed
\end{proof}

Let $f$, $g: \mathbb{N} \times \mathbb{N} \mapsto \mathbb{C}$ be arithmetic functions of two variables.
The Dirichlet convolution of $f$ and $g$ is defined by
\begin{equation}
\nonumber
(f*g)(q,n)=\sum_{d_1 | q, \ d_2|n} f(d_1 , d_2) g( q/d_1,n/d_2 ). 
\end{equation}

Let $f:\mathbb{N} \mapsto \mathbb{C}$ be an arithmetic function. We note that, if we define $f \otimes \delta$ , $\delta \otimes f : \mathbb{N} \times \mathbb{N} \mapsto \mathbb{C}$ by\begin{align*}
(f \otimes \delta) (q,n)=&f(q) \delta(n), \\
(\delta \otimes f) (q,n)=&\delta(q) f(n),
\end{align*} 
then we have for $g: \mathbb{N} \times \mathbb{N} \mapsto \mathbb{C}$
\begin{align*}
(f \underset{\ell}{*} g) (q,n) & =((f \otimes \delta )* g) (q,n) , \\
(g \underset{r}{*} f) (q,n) & =(g*(\delta \otimes f)) (q,n).
\end{align*}

We say that $f : \mathbb{N}  \mapsto \mathbb{C}$ is a multiplicative function if $f$ satisfies\begin{equation}
\nonumber
f(n_1  n_2)=f(n_1) f(n_2) 
\end{equation}
for any $ n_1, n_2 \in \mathbb{N} $ satisfying $ (n_1, n_2)=1.$
It is well known that if $f$ and $g$ are multiplicative functions, then $f*g$ also becomes a multiplicative function. 
We say that $f : \mathbb{N} \times \mathbb{N} \mapsto \mathbb{C}$ is a multiplicative function of two variables if $f$ satisfies\begin{equation}
\nonumber
f(q_1 q_2, n_1  n_2)=f(q_1 , n_1 )  \ f( q_2,n_2) 
\end{equation}
for any $q_1 ,q_2,  n_1, n_2 \in \mathbb{N} $ satisfying $ (q_1 n_1, \  q_2 n_2)=1.$
It is well known that if $f$ and $g$ are multiplicative functions of two variables, then $f*g$ also becomes a multiplicative function of two variables. 

It is easy to see that the following lemma holds.


\begin{lem}
\label{lem2-3}
Let $f,h: \mathbb{N} \mapsto \mathbb{C}$ be multiplicative functions and let $g: \mathbb{N} \times \mathbb{N} \mapsto \mathbb{C}$ be a multiplicative function of two variables.
Then $f \underset{\ell}{*} g$,  \  $g \underset{r}{*} h$   and  $f \underset{\ell}{*} g \underset{r}{*} h$ are all multiplicative functions of two variables.
\end{lem}
\begin{proof}
If $f$ is multiplicative, then $f \otimes \delta$ is also multiplicative as a function of two variables. Therefore $f \underset{\ell}{*} g=(f \otimes \delta )* g$ is also multiplicative as a function of two variables. Similarly, $g \underset{r}{*} h$   and  $f \underset{\ell}{*} g \underset{r}{*} h$ are multiplicative functions of two variables. \qed
\end{proof}

Ramanujan \cite{Ramanujan} proved that $c_q(n)$ can be written as 
\begin{equation}
\nonumber
c_q(n)=\sum_{d \mid (q,n)} \mu (q/d) d. 
\end{equation}
We show that  $c_q(n)$ can also be written as follows.


\begin{lem}
\label{lem2-4}
\begin{equation}
\nonumber
c_q (n)=(\mu \underset{\ell}{*} D \underset{r}{*} \bold{1})(q,n),
\end{equation}
where $\bold{1}(n)=1 $ for every $n \in \mathbb{N}$.
\end{lem}
\begin{proof}
By definition, we have
\begin{equation}
\nonumber
(\mu \underset{\ell}{*} D \underset{r}{*} \bold{1} )(q,n) =\sum_{d_1 | q, \ d_2 | n} \mu(q/d_1) d_1 \delta (d_1,d_2) \bold{1} (n/d_2)=\sum_{d | (q,n)} \mu(q/d) d.        
\end{equation}
\qed
\end{proof}

Hardy \cite{Hardy} proved that, for fixed $n$, $q \mapsto c_q(n)$ is a multiplicative function. Johnson \cite{Johnson} proved that $(q,n) \mapsto c_q(n)$ is a multiplicative function of two variables. 
We remark that the multiplicativity of $(q,n) \mapsto c_q(n)$ is trivial from Lemma \ref{lem2-4} and Lemma \ref{lem2-3} since $D: \mathbb{N} \times \mathbb{N} \mapsto \mathbb{C}$ is multiplicative as a function of two variables.

It is well known that the following holds (\cite{Sivaramakrishnan}).
For a fixed integer $k$,
\begin{equation}
\nonumber
\sum_{q \mid k} c_q(n)=I_{k \mid n} k =\left\{ \begin{array}{cl} 
  k  & \mathrm{if} \ \ k \mid n  \\ 0 &  \mathrm{if}  \ \ k \nmid n.  \end{array}  \right.
\end{equation}
We give another expression of the above in the following lemma. We simply write $c$ instead of $c_{\cdot} (\cdot)$. 
We note that $\sum_{q \mid k} c_q(n)$ can be written as $(\bold{1} \underset{\ell}{*}  c )(k,n) $.


\begin{lem}
\label{lem2-5}
We have
\begin{alignat*}{3}
&(\bold{1} \underset{\ell}{*}  c )(q,n) =q I_{q | n} &= & \left\{ \begin{array}{lllll} q  &  \mathrm{if}  & q \mid n \\ 0 &  \mathrm{if}  & q \nmid n,  \end{array}  \right.  \quad \quad  \mathrm{and} \\
&(c \underset{r}{*} \mu)(q,n) =n I_{n | q} \mu (\frac{q}{n}) &= & \left\{ \begin{array}{lllll} n \mu (\frac{q}{n}) & \mathrm{if} & n \mid q \\ 0 & \mathrm{if}  & n \nmid q.  \end{array}  \right.  
\end{alignat*}
\end{lem}
\begin{proof}
Since $c=\mu \underset{\ell}{*} D \underset{r}{*} \bold{1}$ and $\bold{1} * \mu=\delta $, we have by Lemma \ref{lem2-2}
\begin{align*}
(\bold{1} \underset{\ell}{*}  c )(q,n) & = (\bold{1} \underset{\ell}{*} (\mu \underset{\ell}{*} D \underset{r}{*} \bold{1}))(q,n)= ((\bold{1} * \mu) \underset{\ell}{*} (D \underset{r}{*} \bold{1}))(q,n)=(D \underset{r}{*} \bold{1})(q,n)=q I_{q | n}  ,  \ \ \mathrm{and} \\
(c \underset{r}{*} \mu)(q,n) & = ((\mu \underset{\ell}{*} D \underset{r}{*} \bold{1}) \underset{r}{*} \mu )(q,n)= ((\mu \underset{\ell}{*} D) \underset{r}{*}( \bold{1} * \mu ))(q,n)=(\mu \underset{\ell}{*} D)(q,n)=n I_{n | q} \mu (\frac{q}{n}) .
\end{align*}
\qed
\end{proof}


\section{Some Results}

In this section we show some results concerning dual Ramanujan-Fourier series.
First, we introduce the following Lucht's theorem concerning Ramanujan-Fourier series.

\begin{thm}[Lucht \cite{Lucht}]
\label{th:Lucht}
Let $a:\mathbb{N} \mapsto \mathbb{C}$ be an arithmetic function. If the series 
\begin{equation}
\nonumber
A(n):=n \sum_{k=1}^\infty \mu (k) a(kn) 
\end{equation}
converges for every $n \in \mathbb{N}$, then for $f(n)=(A* \bold{1})(n)$, we have
\begin{equation}
\nonumber
f(n)=\sum_{q=1}^\infty a(q) c_q (n).
\end{equation}
\end{thm}

Lucht obtained (\ref{eq:tau}) by taking $a(n)=-\frac{\log n}{n}$. In this case, we see that 
\begin{equation}
A(n)=n \sum_{k=1}^\infty \mu (k) a(kn) =- n \sum_{k=1}^\infty \mu (k) \frac{\log k+\log n}{kn}=\bold{1} (n),
\end{equation}
since $\sum_{k=1}^\infty \frac{\mu (k) \log k}{k}=-1$ and $\sum_{k=1}^\infty \frac{\mu (k) }{k}=0$.
Therefore $f=A*\bold{1}=\bold{1} * \bold{1}=\tau $ satisfies (\ref{eq:tau}).
We would like to extend Lucht's theorem to the case of dual Ramanujan-Fourier series.
We show the following theorem which is "dual" to Lucht's theorem.


\begin{thm}
\label{th3-2}
Let $a:\mathbb{N} \mapsto \mathbb{C}$ be an arithmetic function. If the series 
\begin{equation}
\nonumber
A(q):=q \sum_{k=1}^\infty a(kq) 
\end{equation}
converges for every $ q \in \mathbb{N} $, then for $ f(q)=(A* \mu)(q) $, we have
\begin{equation}
\label{eq:th3-2}
f(q)=\sum_{n=1}^\infty a(n) c_q (n).
\end{equation}
\end{thm}

\begin{proof}
Since $c_q(n)=(\mu \underset{\ell}{*} D \underset{r}{*}  \bold{1} )(q,n)$, we have by Lemma \ref{lem2-2}
\begin{align*}
\sum_{n \leqq x} a(n) c_q(n)  = & \sum_{n \leqq x} a(n) ((\mu \underset{\ell}{*} D) \underset{r}{*}  \bold{1} )(q,n)= \sum_{n \leqq x} a(n) \sum_{d|n} (\mu  \underset{\ell}{*}  D)(q,d) \bold{1} (\frac{n}{d}) \\
= & \sum_{dk \leqq x} a(dk) (\mu  \underset{\ell}{*}  D) (q,d) =\sum_{dk \leqq x} a(dk) I_{d|q} \mu( \frac{q}{d}) d \\
= &  \sum_{d \leqq x} I_{d|q} \mu(\frac{q}{d}) d \sum_{k \leqq x/d} a(dk) =\sum_{d|q} \mu(\frac{q}{d}) (d \sum_{k \leqq x/d} a(dk)) ,
\end{align*}
where $x$ is a sufficiently large real number. Letting $x  \to   \infty$, we have 
\begin{equation}
\nonumber
\sum_{n =1}^\infty a(n) c_q(n) =\sum_{d|q}  \mu(\frac{q}{d})  A(d)=(\mu * A)(q) =f(q),
\end{equation}
which proves Theorem \ref{th3-2}. \qed
\end{proof}

\begin{rem}
\label{rem3-01}
It is easy to see that, if we set $a(n)=1/n^s$ \  $(s>1)$, then we obtain (\ref{eq:ra-1}) since
\begin{align*}
A(q)=q \sum_{k=1}^\infty a(kq)=q \sum_{k=1}^\infty \frac{1}{(kq)^s}=\zeta(s) \frac{1}{q^{s-1}}= \zeta(s) \mathrm{id}^{1-s} (q) .\end{align*}
\end{rem}

We show other examples of Theorem \ref{th3-2} below.


\begin{exa}
\label{ex3-1}
Let $\omega (q)$ denote the number of distinct prime divisors of $q$ and let $\lambda (q)=(-1)^{\Omega(q)}$ be the Liouville function where $\Omega(q)$ is the number of prime factors of $q$, counted with multiplicity. Then we have
\begin{equation}
\nonumber
 2^{\omega (q)} \lambda (q) = -\frac{1}{\zeta(2)} \sum_{n=1}^\infty \frac{\lambda (n) \log n}{n} c_q (n).  
\end{equation}
\end{exa}
\begin{proof}
Let $a(n)=\lambda (n) \log n /n$. Then, noting that $\lambda$ is completely multiplicative, we have
\begin{align*}
A(q) & =q \sum_{k=1}^\infty a(kq)=q \sum_{k=1}^\infty \frac{\lambda (kq) (\log k + \log q)}{kq} \\
 &=\lambda (q)  \sum_{k=1}^\infty  \frac{\lambda (k) (\log k + \log q)}{k} =-\zeta(2) \lambda(q) ,
\end{align*}
since $\sum_{k=1}^\infty \frac{\lambda (k) \log k}{k}=-\zeta(2)$ and $\sum_{k=1}^\infty \frac{\lambda (k) }{k}=0$.
Therefore, if we set
\begin{equation}
\nonumber
f(q)=(A* \mu )(q)=-\zeta (2) (\lambda * \mu)(q)=-\zeta (2) 2^{\omega (q)} \lambda (q), 
\end{equation}
then Theorem \ref{th3-2} gives the desired result. \qed
\end{proof}


\begin{exa}
\label{ex3-2}
Let $s>1.$ Then we have
\begin{equation}
\label{eq:ex3-2}
\frac{1}{\zeta (s)} \Bigl(\frac{(\mathrm{id} ) \mu}{\varphi_{s}  } * \mu \Bigr) (q)=\sum_{n=1}^\infty \frac{\mu(n)}{n^s} c_q (n).  
\end{equation}
\end{exa}
\begin{proof}
Setting $a(n)=\mu(n) /n^s$, we have
\begin{align*}
A(q) & =q \sum_{k=1}^\infty a(kq)=q \sum_{k=1}^\infty \frac{\mu(kq)}{(kq)^s}=q \sum_{\substack{k\geqq 1 \\ (k,q)=1}} \frac{\mu(k) \mu(q)}{k^s q^s} \\
 &=\frac{q \mu(q)}{q^s} \prod_{ p \nmid q} (1-\frac{1}{p^s})= \frac{q \mu(q)}{q^s} \frac{1}{\zeta(s) \prod_{ p \mid q} (1-1/p^s)} \\
& = \frac{1}{\zeta(s)} \frac{\mathrm{id} (q) \mu(q)}{\varphi_s (q)}.
\end{align*}
Theorefore if we set $f=A*\mu=\frac{1}{\zeta(s)} \frac{(\mathrm{id}) \mu}{\varphi_s } * \mu$, then we see that (\ref{eq:ex3-2}) holds by Theorem \ref{th3-2}. \qed
\end{proof}

Ramanujan \cite{Ramanujan} proved
\begin{equation}
\label{ram3-2}
\sum_{q=1}^\infty \frac{c_q(n)}{q}=0,
\end{equation}
which is well known to be equivalent to the prime number theorem.
We remark that, letting $s \to 1$ in (\ref{eq:ex3-2}), we obtain
\begin{equation}
\label{ram3-3}
\sum_{n=1}^\infty \frac{\mu (n) c_q(n)}{n}=0,
\end{equation}
which is "dual" to (\ref{ram3-2}).

We also have another example which is "dual" to (\ref{ram3-2}).


\begin{exa}
\label{ex3-3}
We have
\begin{equation}
\label{eq:ex3-3}
\sum_{n=1}^\infty \frac{\lambda (n) c_q(n)}{n}=0.
\end{equation}
\end{exa}
\begin{proof}
Taking $a(n)=\lambda(n)/n$ in Theorem \ref{th3-2}, we have
\begin{align*}
A(q)= & q \sum_{k=1}^\infty a(kq)=q \sum_{k=1}^\infty \frac{\lambda (kq) }{kq} = \lambda (q)  \sum_{k=1}^\infty  \frac{\lambda (k) }{k} =0,
\end{align*}
from which $f=A*\mu=0$ follows. \qed
\end{proof}

Next we introduce Delange's theorem concerning Ramanujan-Fourier series.
Given an arithmetic function $a:\mathbb{N} \mapsto \mathbb{C}$, it is convenient to use Theorem \ref{th3-2} in order to find $f$ satisfying (\ref{eq:th3-2}). However, given $f$, it is not convenient to use Theorem \ref{th3-2} in order to find $a$ satisfying (\ref{eq:th3-2}). In the case of Ramanujan-Fourier series, it is sometimes useful to use the following Delange's theorem in order to find $a$ satisfying (\ref{eq1-1}) for given $f$. We will extend Delange's theorem to the case of dual Ramanujan-Fourier series later.


\begin{thm}[Delange \cite{Delange}]
\label{th:Delange}
Let $f(n)$ be an arithmetic function satisfying
\begin{equation}
\label{eq:Del-1}
\sum_{n=1}^\infty 2^{\omega(n)} \frac{|(f * \mu)(n)|}{n} < \infty.
\end{equation}
Then its Ramanujan-Fourier series is pointwise convergent and
\begin{equation}
\nonumber
f(n)=\sum_{q=1}^\infty a(q) c_q (n)
\end{equation}
holds where 
\begin{equation}
\nonumber
a(q)=\sum_{m=1}^\infty \frac{(f * \mu)(qm)}{qm} .  
\end{equation}
Moreover, if $f$ is a multiplicative function, then $a(q)$ can be rewritten as
\begin{equation}
\label{eq:Del-2}
a(q)=\prod_{p \in \mathcal{P}} \Bigl( \sum_{e=\nu_p(q)}^\infty \frac{(f * \mu)(p^e)}{p^e} \Bigr),
\end{equation}
where $\mathcal{P}$ denotes the set of prime numbers and $\nu_p(q)=\left\{ \begin{array}{lll} \alpha & \mathrm{if} & p^\alpha || q \\ 0 & \mathrm{if} & p \nmid q . \end{array}  \right.   $
\end{thm}

Lucht \cite{Lucht} showed that Theorem \ref{th:Delange} can easily be obtained from Theorem \ref{th:Lucht}.
We would like to extend Theorem \ref{th:Delange} to the case of dual Ramanujan-Fourier series by using Theorem \ref{th3-2}. 


\begin{thm}
\label{th3-4}
Let $f$ be an arithmetic function satisfying
\begin{equation}
\label{eq1:th3-4}
\sum_{q=1}^\infty  \frac{|(f * \bold{1})(q)|}{q} \tau(q) < \infty.
\end{equation}
Then its dual Ramanujan-Fourier series is pointwise convergent and
\begin{equation}
\nonumber
f(q)=\sum_{n=1}^\infty a(n) c_q (n)
\end{equation}
holds where 
\begin{equation}
\label{eq2:th3-4}
a(n)=\sum_{m=1}^\infty \frac{(f * \bold{1})(nm)}{nm} \mu(m).  
\end{equation}
\end{thm}


\begin{rem}
\label{rem3-2}
If $f$ is a multiplicative function satisfying
\begin{equation}
\label{eq3:th3-4}
\sum_{p \in \mathcal{P}} \sum_{e=1}^\infty \frac{|(f* \bold{1}) (p^e)|}{p^e} (e+1) < \infty,
\end{equation}
then its dual Ramanujan-Fourier series is pointwise convergent and
\begin{equation}
\nonumber
f(q)=\sum_{n=1}^\infty a(n) c_q (n)
\end{equation}
holds where 
\begin{equation}
\nonumber
a(n)=\sum_{m=1}^\infty \frac{(f * \bold{1})(nm)}{nm} \mu(m).  
\end{equation}
Moreover, $a(n)$ can be rewritten as
\begin{equation}
\label{eq4:th3-4}
a(n)=\prod_{p \in \mathcal{P}} \Bigl(  \frac{(f * \bold{1})(p^{\nu_p(n)})}{p^{\nu_p(n)}}-\frac{(f * \bold{1})(p^{\nu_p(n)+1})}{p^{\nu_p(n)+1}} \Bigr).
\end{equation}
\end{rem}


\begin{proof}[Proof of Theorem \ref{th3-4}.]
We first see that $A(q)=q \sum_{k=1}^\infty a(kq)$ converges since
\begin{align*}
\sum_{k \leqq x } |a(kq)| & \leqq \sum_{k \leqq x } \sum_{m=1}^\infty \frac{|(f * \bold{1})(kqm)|}{kqm} |\mu(m)| \\
 & \leqq \sum_{\ell=1}^\infty  \frac{|(f* \bold{1}) (\ell)|}{\ell} \sum_{k \leqq x, \ k| \ell} 1 \leqq \sum_{\ell=1}^\infty  \frac{|(f* \bold{1}) (\ell)|}{\ell} \tau(\ell) < \infty
\end{align*}
holds for every $x>1$. Using Lemma \ref{lem2-5} we can rewrite $A(q)$ as follows.
\begin{align*}
A(q) &=q \sum_{k=1}^\infty a(kq)= \sum_{m=1}^\infty a(m) q I_{q|m} = \sum_{m=1}^\infty a(m) ( \bold{1} \underset{\ell}{*} c)(q,m).
\end{align*}
From this we have
\begin{align*}
f(q) &=(\mu * A)(q)=\sum_{m=1}^\infty a(m) (\mu \underset{\ell}{*} ( \bold{1} \underset{\ell}{*} c))(q,m)=\sum_{m=1}^\infty a(m) ((\mu * \bold{1}) \underset{\ell}{*} c)(q,m) \\
 &=\sum_{m=1}^\infty a(m) (\delta \underset{\ell}{*} c)(q,m) =\sum_{m=1}^\infty a(m) c_q (m).
\end{align*}
which completes the proof of Theorem \ref{th3-4}. \qed
\end{proof}


\begin{proof}[Proof of Remark \ref{rem3-2}.]
If $f$ is a multiplicative function, then $q \mapsto \frac{(f * \bold{1})(q)}{q} \tau(q)$ is also a multiplicative function. Using $1+x \leqq \exp(x)$, we see that (\ref{eq1:th3-4}) follows from (\ref{eq3:th3-4}) since 
\begin{align*}
\sum_{q \leqq Q}  \frac{|(f * \bold{1})(q)|}{q} \tau(q) \leqq & \prod_{p \in \mathcal{P}} \Bigl(1+\sum_{e=1}^\infty \frac{|(f * \bold{1})(p^e)|}{p^e} \tau(p^e) \Bigr) \\
 \leqq & \prod_{p \in \mathcal{P}} \exp \Bigl( \sum_{e=1}^\infty \frac{|(f * \bold{1})(p^e)|}{p^e} \tau(p^e) \Bigr)=\exp \Bigl( \sum_{p \in \mathcal{P}} \sum_{e=1}^\infty \frac{|(f * \bold{1})(p^e)|}{p^e} (e+1) \Bigr) < \infty
\end{align*}
holds for every $Q>1$. Therefore (\ref{eq2:th3-4}) holds by Theorem \ref{th3-4}.  In the expression of (\ref{eq2:th3-4}), we set
 $n=\prod p_j^{e_j} $,  $m=r \prod p_j^{d_j}$  where  $(r,n)=1$,  $e_j \geqq 1$, and $d_j \geqq 0$.
Then we have
\begin{align*}
a(n)= & \sum_{d_j \geqq 0, \ r \geqq 1, \ (r,n)=1} \frac{(f*\bold{1}) (r \prod p_j^{e_j+d_j})}{r \prod p_j^{e_j+d_j}} \mu(r \prod p_j^{d_j}) .
\end{align*}
Since $f*\bold{1}$ is multiplicative and since $\mu( p_j^{d_j}) = 0$ if $d_j \geqq 2$ for some $j$, we obtain
\begin{align*}
 a(n)= & \prod_j \Bigl( \sum_{0 \leqq d_j \leqq 1} \frac{(f*\bold{1}) ( p_j^{e_j+d_j})}{ p_j^{e_j+d_j}} \mu(  p_j^{d_j}) \Bigr) \times \sum_{r \geqq 1, \ (r,n)=1} \frac{(f*\bold{1}) (r )}{r } \mu(r) \\
= & \prod_{p|n} \Bigl(  \frac{(f * \bold{1})(p^{\nu_p(n)})}{p^{\nu_p(n)}}-\frac{(f * \bold{1})(p^{\nu_p(n)+1})}{p^{\nu_p(n)+1}} \Bigr) \times \prod_{ p \nmid n} \Bigl( 1-\frac{(f * \bold{1})(p)}{p} \Bigr) \\
= & \prod_{p \in \mathcal{P}} \Bigl(  \frac{(f * \bold{1})(p^{\nu_p(n)})}{p^{\nu_p(n)}}-\frac{(f * \bold{1})(p^{\nu_p(n)+1})}{p^{\nu_p(n)+1}} \Bigr),
\end{align*}
which completes the proof of Remark \ref{rem3-2}. \qed
\end{proof}

Several examples are shown below.


\begin{exa}
\label{ex3-4}
Let $s>1.$ Then we have
\begin{equation}
\nonumber
 \frac{\varphi_s (q)}{q^s} \mu (q)= \frac{1}{\zeta(s+1)} \sum_{n=1}^\infty  \frac{\varphi (K(n))}{n \psi_{s+1} (K(n))} c_q (n) ,
\end{equation}
where $K(n)=\prod_{p|n}p$ and $\psi_s (n)=n^s \prod_{p|n} (1+1/p^s) $.
\end{exa}
\begin{proof}
Let $f(q)=\frac{\varphi_s (q)}{q^s} \mu (q)$. Then it is easy to see that 
\begin{equation}
\nonumber
f(p^e)=\left\{ \begin{array}{cl}  -(1-1/p^s)    & \mathrm{if} \ \ e=1   \\ 0 &  \mathrm{if}  \ \ e \geqq 2  \end{array}  \right.
\end{equation}
and
\begin{equation}
\nonumber
 (f* \bold{1} )(p^e)=\left\{ \begin{array}{cl} 1    & \mathrm{if} \ \ e=0   \\ 1/p^s &  \mathrm{if}  \ \ e \geqq 1,  \end{array}  \right.
\end{equation}
from which we see that (\ref{eq3:th3-4}) holds. It is also easy to see that  
\begin{equation}
\nonumber
 \frac{(f*\bold{1})(p^e)}{p^e}-\frac{(f*\bold{1})(p^{e+1})}{p^{e+1}} =\left\{ \begin{array}{cl} 1-1/p^{s+1}    & \mathrm{if} \ \ e=0  \\ \frac{1}{p^{e+s}} (1-1/p) &  \mathrm{if}  \ \ e \geqq  1. \end{array}  \right.
\end{equation}
From this and (\ref{eq4:th3-4}) we have
\begin{align*}
a(n) & = \prod_{p \mid n} \Bigl(  \frac{(f * \bold{1})(p^{\nu_p(n)})}{p^{\nu_p(n)}}-\frac{(f * \bold{1})(p^{\nu_p(n)+1})}{p^{\nu_p(n)+1}} \Bigr) \prod_{p \nmid n} \Bigl(  \frac{(f * \bold{1})(p^{\nu_p(n)})}{p^{\nu_p(n)}}-\frac{(f * \bold{1})(p^{\nu_p(n)+1})}{p^{\nu_p(n)+1}} \Bigr) \\
& = \prod_{p \mid n} \frac{1}{p^{\nu_p (n)+s}} (1-1/p) \prod_{p \nmid n} (1-1/p^{s+1}) \\
 & = \prod_{p \mid n} \frac{1}{p^{\nu_p (n)+s}} \frac{1-1/p}{1-1/p^{s+1}} \prod_{p \in \mathcal{P}} (1-1/p^{s+1}) \\
 & = \frac{1}{\zeta(s+1)} \prod_{p \mid n} \frac{1}{p^{\nu_p (n)}} \frac{p(1-1/p)}{p^{s+1}(1-1/p^{s+1})}  = \frac{1}{\zeta(s+1)} \frac{\varphi (K(n))}{n \psi_{s+1} (K(n))}.
\end{align*}
\qed
\end{proof}


\begin{exa}
\label{ex3-5}
Let $s>1.$ Then we have
\begin{equation}
\nonumber
\frac{\sigma_s (q)}{q^s} \mu (q)=\frac{\zeta(s+1)}{\zeta(2s+2)} \sum_{n=1}^\infty   \frac{(-1)^{\omega(n) } \varphi (K(n))}{n \psi_{s+1}(K(n))}  c_q (n) .
\end{equation}

\end{exa}
\begin{proof}
Let $f(q)=\frac{\sigma_s (q)}{q^s} \mu (q)$. Then it is easy to see that 
\begin{equation}
\nonumber
f(p^e)=\left\{ \begin{array}{cl}   -1-1/p^s  & \mathrm{if} \ \ e=1  \\ 0 &  \mathrm{if}  \ \ e \geqq 2  \end{array}  \right.
\end{equation}
and
\begin{equation}
\nonumber
(f* \bold{1} )(p^e)=\left\{ \begin{array}{cl} 1    & \mathrm{if} \ \  e=0 \\   -1/p^s & \mathrm{if}  \ \ e \geqq 1,  \end{array}  \right.
\end{equation}
from which we see that (\ref{eq3:th3-4}) holds. It is also easy to see that  
\begin{equation}
\nonumber
 \frac{(f*\bold{1})(p^e)}{p^e}-\frac{(f*\bold{1})(p^{e+1})}{p^{e+1}} =\left\{ \begin{array}{cl} 1+1/p^{s+1}  & \mathrm{if} \ \  e=0 \\ \frac{-1}{p^{e+s}} (1-1/p) &  \mathrm{if}  \ \ e \geqq 1.   \end{array}  \right.
\end{equation}
Therefore by (\ref{eq4:th3-4}) we have 
\begin{align*}
a(n) & = \prod_{p \mid n} \Bigl(  \frac{(f * \bold{1})(p^{\nu_p(n)})}{p^{\nu_p(n)}}-\frac{(f * \bold{1})(p^{\nu_p(n)+1})}{p^{\nu_p(n)+1}} \Bigr) \prod_{p \nmid n} \Bigl(  \frac{(f * \bold{1})(p^{\nu_p(n)})}{p^{\nu_p(n)}}-\frac{(f * \bold{1})(p^{\nu_p(n)+1})}{p^{\nu_p(n)+1}} \Bigr) \\
 & = \prod_{p \mid n} \frac{-1}{p^{\nu_p (n)+s}} (1-1/p) \prod_{p \nmid n} (1+1/p^{s+1}) \\
 & = \prod_{p \mid n} \frac{-1}{p^{\nu_p (n)+s}} \frac{1-1/p}{1+1/p^{s+1}} \prod_{p \in \mathcal{P}} (1+1/p^{s+1}) \\
 & =\frac{\zeta(s+1)}{\zeta(2s+2)}   \prod_{p \mid n} \frac{-1}{p^{\nu_p (n)} } \frac{p (1-1/p)}{p^{s+1} (1+1/p^{s+1})} =\frac{\zeta(s+1)}{\zeta(2s+2)} \frac{(-1)^{\omega(n) } \varphi (K(n))}{n \psi_{s+1}(K(n))}  .
\end{align*}
\qed
\end{proof}


\begin{exa}
\label{ex3-6}
We have 
\begin{equation}
\nonumber
\lambda(q)=\sum_{n=1}^\infty  a(n)  c_q (n),
\end{equation}
where $a(n)=\displaystyle{\frac{1}{n} \prod_{\substack{p \mid n \\ \nu_p(n):odd}} \frac{-1}{p}}  $ .
\end{exa}
\begin{proof}
Let $f(q)=\lambda (q)$. Then it is easy to see that 
\begin{equation}
\nonumber
(f* \bold{1} )(p^e)=\begin{cases} 0 \ \ \mathrm{if} \ \ e \ \ \mathrm{is \  odd} \\ 1 \ \ \mathrm{if}  \ \ e \ \ \mathrm{is   \ even}, \end{cases} 
\end{equation}
from which we see that (\ref{eq3:th3-4}) holds. It is also easy to see that  
\begin{equation}
\nonumber
 \frac{(f*\bold{1})(p^e)}{p^e}-\frac{(f*\bold{1})(p^{e+1})}{p^{e+1}} =\left\{ \begin{array}{ll}  -1/p^{e+1} & \mathrm{if}  \ \ e \ \mathrm{is \ odd}  \\ 1/p^{e}  &  \mathrm{if}  \ \ e   \ \mathrm{is \ even} .  \end{array}  \right.
\end{equation}
Therefore by (\ref{eq4:th3-4}) we have
\begin{align*}
a(n) & = \prod_{p \mid n} \Bigl(  \frac{(f * \bold{1})(p^{\nu_p(n)})}{p^{\nu_p(n)}}-\frac{(f * \bold{1})(p^{\nu_p(n)+1})}{p^{\nu_p(n)+1}} \Bigr) \prod_{p \nmid n} \Bigl(  \frac{(f * \bold{1})(p^{\nu_p(n)})}{p^{\nu_p(n)}}-\frac{(f * \bold{1})(p^{\nu_p(n)+1})}{p^{\nu_p(n)+1}} \Bigr) \\
 & = \prod_{\substack{p \mid n \\ \nu_p(n):odd}} \frac{-1}{p^{\nu_p (n)+1}} \prod_{\substack{p \mid n \\ \nu_p(n):even}} \frac{1}{p^{\nu_p (n)}} \prod_{p \nmid n} 1 \\
 & = \prod_{p \mid n } \frac{1}{p^{\nu_p (n)}} \prod_{\substack{p \mid n \\ \nu_p(n):odd}} \frac{-1}{p}=\frac{1}{n} \prod_{\substack{p \mid n \\ \nu_p(n):odd}} \frac{-1}{p} .
\end{align*}
\qed
\end{proof}


\begin{exa}
\label{ex3-7}
\begin{equation}
\nonumber
\frac{\varphi (q)}{q} \lambda(q)=\frac{1}{\zeta(2)} \sum_{n=1}^\infty   \frac{I_{square}(n)}{n}  c_q (n) ,
\end{equation}
where $I_{square}(n) =\begin{cases}  1 \ \ \mathrm{if}  \ \ n  \ \ \mathrm{is \ a \ perfect \ square} \\ 0  \ \ \mathrm{otherwise} . \end{cases} $
\end{exa}
\begin{proof}
Let $f(q)=\frac{\varphi (q)}{q} \lambda(q)$. Then it is easy to see that 
\begin{equation}
\nonumber
(f* \bold{1} )(p^e)=\left\{ \begin{array}{cl} 1/p    & \mathrm{if} \ \  e \ \mathrm{is \  odd} \\ 1 &  \mathrm{if}  \ \  e \ \mathrm{is   \ even}.  \end{array}  \right. 
\end{equation}
from which we see that (\ref{eq3:th3-4}) holds. It is also easy to see that  
\begin{equation}
\nonumber
 \frac{(f*\bold{1})(p^e)}{p^e}-\frac{(f*\bold{1})(p^{e+1})}{p^{e+1}} =\left\{ \begin{array}{cl}  0   & \mathrm{if} \ \ e   \ \mathrm{is \ odd}  \\ \frac{1}{p^{e}} (1-1/p^2)  &  \mathrm{if}  \ \  e  \ \mathrm{is \ even}.   \end{array}  \right.
\end{equation}
Therefore, by (\ref{eq4:th3-4}), $a(n)=0$ if $n$ is not a perfect square. If $n$ is a perfect square, then we have 
\begin{align*}
a(n)  = \prod_{p \mid n} \frac{1}{p^{\nu_p (n)}} (1-1/p^2) \prod_{p \nmid n} (1-1/p^2)  = \prod_{p \in \mathcal{P}} (1-1/p^2) \prod_{p \mid n} \frac{1}{p^{\nu_p (n)}}     = \frac{1}{\zeta(2)} \frac{1}{n} .
\end{align*}
Thus we can express $a(n)$ as
\begin{equation}
\nonumber
a(n)=\frac{1}{\zeta(2)} \frac{I_{square}(n)}{n}
\end{equation}
whether $n$ is a perfect square or not. This completes the proof of Example \ref{ex3-7}.
\qed
\end{proof}

Let $\mathcal{F}$ be the set of real valued arithmetic functions and let $ \mathcal{A}=\{ a \in \mathcal{F} :  \sum_q a(q) c_q (n) \mathrm{converges.} \} $, $ \mathcal{B}=\{ b \in \mathcal{F} :  \sum_n b(n) c_q (n) \mathrm{converges.} \}. $ 
If we define $T: \mathcal{A} \mapsto \mathcal{F}$ and $T^{*}: \mathcal{B} \mapsto \mathcal{F}$ by
\begin{align*}
& (Ta)(n)=\sum_q a(q) c_q (n) , \\
& (T^{*} b)(q)=\sum_n b(n) c_q (n),
\end{align*}
respectively, then we have "formally"
\begin{align*}
<b,Ta>=<T^{*} b,a>,
\end{align*}
where $<b,a>:=\sum_n b(n) a(n)$ is an inner product of $a$ and $b$. More precisely, we have the following trivial proposition.


\begin{pro}
\label{pro3-1}
If $f(n)=\sum_q a(q) c_q(n)$ and $g(q)=\sum_n b(n) c_q(n).$ If 
\begin{equation}
\label{pro3-1}
\sum_{q,n} |a(q) c_q(n) b(n)|<\infty,
\end{equation}
then
\begin{equation}
\nonumber
\sum_n f(n) b(n)=\sum_q a(q) g(q).
\end{equation}
\end{pro}

\begin{proof}
\begin{equation}
\nonumber
\sum_n f(n) b(n)=\sum_{q,n} a(q) c_q(n) b(n)=\sum_q a(q) g(q).
\end{equation}
\qed
\end{proof}

As an example of the above proposition, we show the following example.


\begin{exa}
\label{ex3-8}
\begin{equation}
\nonumber
 \sum_{n=1}^\infty   \frac{\varphi (n) I_{n:square}}{n^2} = \sum_{q=1}^\infty \frac{\mu(q)^2 }{q \psi(q)} .
\end{equation}
\end{exa}
\begin{proof}
By (\ref{eq:varphi}) with $s=1$ and Example \ref{ex3-7}, we have
\begin{align*}
\frac{\varphi (n)}{n} = & \frac{1}{\zeta(2)} \sum_{q=1}^\infty \frac{\mu (q)}{\varphi_{2} (q)} c_q (n) , \\
\frac{\varphi (q)}{q} \lambda(q)= & \frac{1}{\zeta(2)} \sum_{n=1}^\infty   \frac{I_{square}(n)}{n}  c_q (n) .
\end{align*}
We note that the right hand of (\ref{eq:varphi}) is absolutely convergent. Hence (\ref{pro3-1}) holds.
By Proposition \ref{pro3-1} we have
\begin{align*}
\sum_{n=1}^\infty   \frac{\varphi (n)}{n} \frac{ I_{n:square}}{n} = & \sum_{q=1}^\infty \frac{\mu (q)}{\varphi_{2} (q)} \frac{\varphi (q)}{q} \lambda(q) = \sum_{q=1}^\infty \frac{\mu (q) \lambda(q)  \prod_{p \mid  q} (1-1/p)}{ q^2 \prod_{p \mid  q} (1-1/p^2)} \\
 = & \sum_{q=1}^\infty \frac{\mu (q)^2 }{ q^2 \prod_{p \mid  q} (1+1/p)} = \sum_{q=1}^\infty \frac{\mu(q)^2 }{q \psi(q)} ,
\end{align*}
which completes the proof of Example \ref{ex3-8}. \qed
\end{proof}

Of course, Example \ref{ex3-8} can also be obtained by expressing both sides as infinite products by prime numbers.


\begin{rem}
We do not know whether we can loosen the condition (\ref{pro3-1}) or not.
If we can, then, for every $f \in T \mathcal{A}$ such that $f=Ta$ and for every $g \in T^* \mathcal{B}$ such that $g=T^* b$, we have "formally"
\begin{align*}
& \sum_n \frac{f(n) \mu(n)}{n}= <Ta,\frac{\mu}{\mathrm{id}}>=  <a,T^* \displaystyle{\frac{\mu}{\mathrm{id}}}>=<a,0>=0 , \\
& \sum_q \frac{g(q) }{q}= <T^* b,\frac{1}{\mathrm{id}}>=  <b,T \frac{1}{\mathrm{id}}>=<b,0>=0 ,
\end{align*}
namely, $\mathrm{Im} T \perp \mathrm{Ker} T^*$ and $\mathrm{Im} T^* \perp \mathrm{Ker} T$.
However, we can't prove the above rigorously.
\end{rem}

Next we consider Dirichlet series of a function expressed as Ramanujan-Fourier series or dual Ramanujan-Fourier series. We show the following theorem.


\begin{thm}
\label{th3-5}
Suppose $s>1$. Let $f$ be an arithmetic function such that the Dirichlet series $\sum_{n=1}^\infty \frac{f(n)}{n^s}$ converges absolutely and let $a( \cdot )$ be a multiplicative function such that $\sum_{ k,n \geqq 1} \frac{|a(kn)|}{n^{s-1}} < \infty $. \\
{\rm{(i)}} \ If $f(n)=\sum_{q=1}^\infty a(q) c_q(n)$ converges absolutely, then the Dirichlet series of $f$ is expressed as
\begin{equation}
\nonumber
\sum_{n=1}^\infty \frac{f(n)}{n^s}= \zeta(s) \prod_{ p \in \mathcal{P}} ( \sum_{e \geqq 0} \frac{a(p^e)-a(p^{e+1})}{p^{e(s-1)}}).
\end{equation}
{\rm{(ii)}} \ If $f(q)=\sum_{n=1}^\infty a(n) c_q(n)$ converges absolutely , then the Dirichlet series of $f$ is expressed as
\begin{equation}
\nonumber
\sum_{q=1}^\infty \frac{f(q)}{q^s}= \frac{1}{\zeta(s)} \prod_{ p \in \mathcal{P}} ( \sum_{e_1 \geqq 0} \frac{\sum_{e_2 \geqq e_1} a(p^{e_2})}{p^{e_1 (s-1)}}).
\end{equation}
\end{thm}

\begin{proof}
{\rm{(i)}} 
Since $f(n)=\sum_{q=1}^\infty a(q) c_q(n)= \sum_{q=1}^\infty a(q)  (\mu \underset{\ell}{*} D \underset{r}{*} \bold{1} )(q,n) $, we have by Lemma \ref{lem2-2}
\begin{align*}
(f*\mu )(n)= & \sum_{q=1}^\infty a(q)  ((\mu \underset{\ell}{*} D \underset{r}{*} \bold{1} ) \underset{r}{*} \mu) (q,n) = \sum_{q=1}^\infty a(q)  (\mu \underset{\ell}{*} D) (q,n) \\
 = & \sum_{q=1}^\infty a(q) I_{n \mid q} \mu (\frac{q}{n}) n =\sum_{\substack{q \geqq 1 \\ n \mid q}} a(q)  \mu (\frac{q}{n}) n.
\end{align*}
From this we have
\begin{align*}
\sum_{n=1}^\infty \frac{(f*\mu)(n)}{n^s}= & \sum_{n=1}^\infty \frac{1}{n^s} \sum_{\substack{q \geqq 1 \\ n \mid q}} a(q)  \mu (\frac{q}{n}) n=\sum_{q \geqq 1} a(q) \sum_{n \mid q} \frac{ \mu(\frac{q}{n})  }{n^{s-1} }  = \sum_{q=1}^\infty a(q) (\frac{1}{\mathrm{id}^{s-1}} * \mu)(q) \\
=& \prod_{ p \in \mathcal{P}} ( \sum_{e \geqq 0} a(p^e) (\frac{1}{\mathrm{id}^{s-1}} * \mu)(p^e) ) = \prod_{ p \in \mathcal{P}} (1+ \sum_{e \geqq 1} a(p^e) ( \frac{1}{p^{e(s-1)}} - \frac{1}{p^{(e-1)(s-1)}} )) \\
=&   \prod_{ p \in \mathcal{P}} (1-a(p)+\sum_{e \geqq 1} \frac{a(p^e)-a(p^{e+1})}{p^{e(s-1)}}) =  \prod_{ p \in \mathcal{P}} ( \sum_{e \geqq 0} \frac{a(p^e)-a(p^{e+1})}{p^{e(s-1)}}).
\end{align*}
Therefore we have 
\begin{align*}
\frac{1}{\zeta(s)} \sum_{n=1}^\infty \frac{f(n)}{n^s} =\prod_{ p \in \mathcal{P}} ( \sum_{e \geqq 0} \frac{a(p^e)-a(p^{e+1})}{p^{e(s-1)}}).
\end{align*}

{\rm{(ii)}} 
We proceed in a similar manner. 
Since $f(q)=\sum_{n=1}^\infty a(n) c_q(n)= \sum_{n=1}^\infty a(n)  (\mu \underset{\ell}{*} D \underset{r}{*} \bold{1} )(q,n) $, we have by Lemma \ref{lem2-2}
\begin{align*}
(\bold{1} * f )(q)= & \sum_{n=1}^\infty a(n)  (\bold{1} \underset{\ell}{*} (\mu \underset{\ell}{*} D \underset{r}{*} \bold{1} )) (q,n) = \sum_{n=1}^\infty a(n)  (D \underset{r}{*} \bold{1}) (q,n) \\
 = & \sum_{n=1}^\infty a(n) I_{q \mid n} q =\sum_{\substack{n \geqq 1 \\ q \mid n}} a(n)  q.
\end{align*}
From this we have
\begin{align*}
\sum_{q=1}^\infty \frac{(\bold{1} * f)(q)}{q^s}= & \sum_{q=1}^\infty \frac{1}{q^s} \sum_{\substack{n \geqq 1 \\ q \mid n}} a(n) q=\sum_{n \geqq 1} a(n) \sum_{q \mid n} \frac{ 1  }{q^{s-1} }  \\
= & \sum_{n=1}^\infty a(n) (\frac{1}{\mathrm{id}^{s-1}} * \bold{1})(n) = \prod_{ p \in \mathcal{P}} ( \sum_{e \geqq 0} a(p^e) (\frac{1}{\mathrm{id}^{s-1}} * \bold{1})(p^e) ) \\
 =& \prod_{ p \in \mathcal{P}} ( \sum_{e \geqq 0} a(p^e) (\mathrm{id}^{1-s} * \bold{1})(p^e) ) =  \prod_{ p \in \mathcal{P}} (\sum_{e_2  \geqq 0} a(p^{e_2}) \sum_{0 \leqq e_1  \leqq e_2} \frac{1}{p^{e_1 (s-1)}} ) \\
=&  \prod_{ p \in \mathcal{P}} ( \sum_{e_1 \geqq 0} \frac{\sum_{e_2 \geqq e_1} a(p^{e_2})}{p^{e_1 (s-1)}}).
\end{align*}
Therefore we have 
\begin{align*}
\zeta(s) \sum_{q=1}^\infty \frac{f(q)}{q^s} =\prod_{ p \in \mathcal{P}} ( \sum_{e_1 \geqq 0} \frac{\sum_{e_2 \geqq e_1} a(p^{e_2})}{p^{e_1 (s-1)}}),
\end{align*}
which completes the proof of Theorem \ref{th3-5}. \qed
\end{proof}

As an example of Theorem \ref{th3-5}, we show the following example.


\begin{exa}
\begin{equation}
\nonumber
\frac{\lambda(q) K(q) \psi (q)}{q^2}=\frac{\zeta(2)}{\zeta(4)} \sum_{n=1}^\infty \frac{\lambda(n)}{n^2} c_q (n).
\end{equation}
\begin{proof}
Let $a(n)=\lambda(n)/n^2.$ Then we have 
\begin{align*}
& \sum_{e_2 \geqq e_1} a(p^{e_2})= \frac{(-1)^{e_1}}{p^{2e_1}}+\frac{(-1)^{e_1+1}}{p^{2(e_1+1)}}+ \cdots    =\frac{(-1)^{e_1} }{p^{2 e_1} (1+1/p^2)}, \quad \mathrm{and} \\
& \sum_{e_1 \geqq 0} \frac{\sum_{e_2 \geqq e_1} a(p^{e_2})}{p^{e_1 (s-1)}}= \frac{1}{1+1/p^2}  \sum_{e \geqq 0} \frac{(-1)^e }{p^{e(s+1)}} = \frac{1}{1+1/p^2} \frac{1}{1+1/p^{s+1}} .
\end{align*}
From this we have
\begin{align*}
\frac{1}{\zeta(s)} \prod_{ p \in \mathcal{P}} ( \sum_{e_1 \geqq 0} \frac{\sum_{e_2 \geqq e_1} a(p^{e_2})}{p^{e_1 (s-1)}})= & \frac{1}{\zeta(s)} \prod_{ p \in \mathcal{P}} \frac{1}{1+1/p^2} \frac{1}{1+1/p^{s+1}} = \frac{1}{\zeta(s)} \frac{\zeta(4)}{\zeta(2)} \frac{\zeta(2s+2)}{\zeta(s+1)}.
\end{align*}
Therefore, if we set $f(q)=\sum_{n=1}^\infty a(n) c_q(n)$, then $f$ satisfies
\begin{equation}
\nonumber
\sum_{q=1}^\infty \frac{f(q)}{q^s} =\frac{1}{\zeta(s)} \frac{\zeta(4)}{\zeta(2)} \frac{\zeta(2s+2)}{\zeta(s+1)}.
\end{equation}
On the other hand, if we set $\widetilde{f}(q)= \frac{\lambda(q) K(q) \psi (q)}{q^2}$, then $\widetilde{f}$ satisfies
\begin{align*}
\sum_{q=1}^\infty \frac{\widetilde{f} (q)}{q^s} = & \prod_{p \in \mathcal{P}} (1+\sum_{e \geqq 1} \frac{1}{p^{es}} \frac{\lambda(p^e) K(p^e) \psi (p^e)}{p^{2e}}) =  \prod_{p \in \mathcal{P}} (1+\sum_{e \geqq 1} \frac{1}{p^{es}}  \frac{(-1)^e p \cdot p^e(1+1/p)}{p^{2e}}) \\
=& \prod_{p \in \mathcal{P}} (1+(p+1) \sum_{e \geqq 1} (\frac{-1}{p^{e(s+1)}} ))=\prod_{p \in \mathcal{P}} (1-\frac{p+1}{p^{s+1}+1}) \\
= & \prod_{p \in \mathcal{P}} \frac{1-1/p^s}{1+1/p^{s+1}}=   \frac{\zeta(2s+2)}{\zeta(s) \zeta(s+1)}.
\end{align*}
By the uniqueness of the Dirichlet series, we have $\widetilde{f}(q)=\frac{\zeta(2)}{\zeta(4)} f(q) $, namely
\begin{align*}
\frac{\lambda(q) K(q) \psi (q)}{q^2}= \frac{\zeta(2)}{\zeta(4)}  \sum_{n=1}^\infty \frac{\lambda(n)}{n^2} c_q (n).
\end{align*}
\qed
\end{proof}

\end{exa}


\section{The Case of Arithmetic Functions of Two Variables}

In this section, we consider the case of arithmetic functions of two variables. We would like to extend theorems in section 3 to this case.
In more detail, we consider Ramanujan-Fourier series
\begin{align*}
& f(n_1,n_2)=  \sum_{q_1,q_2=1}^\infty a(q_1,q_2) c_{q_1} (n_1) c_{q_2} (n_2), 
\end{align*}
and dual Ramanujan-Fourier series
\begin{align*}
f(q_1,q_2)= & \sum_{n_1,n_2=1}^\infty a(n_1,n_2) c_{q_1} (n_1) c_{q_2} (n_2),
\end{align*}
where $f,a$ are arithmetic functions of two variables. 

We use the same notations $\bold{1}$ and $\mu$ for the functions 
\begin{align*}
& \bold{1} (n_1,n_2)=\bold{1} (n_1) \bold{1} (n_2), \\
& \mu(n_1, n_2)=\mu(n_1)  \mu(n_2),
\end{align*}
respectively. Clearly, $(\mu * \bold{1}) (n_1, n_2)= \delta(n_1) \delta( n_2)  $ holds. 

We begin with the following theorem which is an extension of Theorem \ref{th:Lucht}.


\begin{thm}
\label{th4-1}
Let $a:\mathbb{N} \times \mathbb{N} \mapsto \mathbb{C} $ be an arithmetic function of two variables. If the series 
\begin{equation}
\nonumber
A(n_1,n_2):=n_1 n_2 \sum_{k_1,k_2=1}^\infty \mu (k_1, k_2) a(k_1 n_1,k_2 n_2) 
\end{equation}
converges for every $n_1,n_2 \in \mathbb{N}$, then for $f(n_1,n_2)=(A* \bold{1})(n_1,n_2)$, we have
\begin{equation}
\nonumber
f(n_1,n_2)=\sum_{q_1,q_2=1}^\infty a(q_1,q_2) c_{q_1} (n_1) c_{q_2} (n_2).
\end{equation}
\end{thm}

\begin{proof} 
Since $c_q(n)=(\mu \underset{\ell}{*} D \underset{r}{*}  \bold{1} )(q,n)$, we have
\begin{align*}
\sum_{\substack{q_1 \leqq x \\ q_2 \leqq y}} a(q_1,q_2) c_{q_1} (n_1) c_{q_2} (n_2)  = & \sum_{\substack{q_1 \leqq x \\ q_2 \leqq y}} a(q_1,q_2) (\mu \underset{\ell}{*} D \underset{r}{*}  \bold{1} )(q_1,n_1) (\mu \underset{\ell}{*} D \underset{r}{*}  \bold{1} )(q_2,n_2) \\
=& \sum_{\substack{q_1 \leqq x \\ q_2 \leqq y}} a(q_1,q_2)  (\sum_{d_1|q_1 } \mu(\frac{q_1}{d_1})(D  \underset{r}{*}  \bold{1})(d_1,n_1))(\sum_{d_2|q_2 } \mu(\frac{q_2}{d_2}) (D  \underset{r}{*}  \bold{1})(d_2,n_2)).
\end{align*}
Setting $q_1=d_1 k_1$, $q_2=d_2 k_2$ and using Lemma \ref{lem2-2}, we see that the above is equal to
\begin{align*}
 & \sum_{\substack{d_1 k_1 \leqq x \\ d_2 k_2 \leqq y}} a(d_1 k_1, d_2 k_2) \mu(k_1) \mu(k_2) (D  \underset{\ell}{*}  \bold{1}) (d_1,n_1) (D  \underset{\ell}{*}  \bold{1}) (d_2,n_2)  \\
= & \sum_{\substack{d_1 k_1 \leqq x \\ d_2 k_2 \leqq y}} a(d_1 k_1, d_2 k_2) \mu(k_1) \mu(k_2)  I_{d_1|n_1} d_1  I_{d_2|n_2} d_2   \\
= &  \sum_{\substack{d_1 \leqq x \\ d_2 \leqq y}} I_{d_1|n_1} I_{d_2 |n_2}   d_1 d_2 \sum_{\substack{k_1 \leqq x/d_1 \\ k_2 \leqq y/d_2}} \mu (k_1, k_2)  a(d_1 k_1, d_2 k_2),
\end{align*}
where $x,y$ are sufficiently large real numbers. Letting $x,y  \to   \infty$, we have 
\begin{align*}
 \sum_{q_1,q_2=1}^\infty a(q_1,q_2) c_{q_1} (n_1) c_{q_2} (n_2)=& \sum_{d_1 , d_2 =1}^\infty I_{d_1|n_1} I_{d_2 |n_2}  ( d_1 d_2 \sum_{k_1, k_2 =1}^\infty \mu (k_1, k_2)  a(d_1 k_1, d_2 k_2)) \\
=& \sum_{\substack{d_1|n_1 \\ d_2|n_2} }   A(d_1,d_2) =(A* \bold{1})(n_1,n_2) =f(n_1,n_2),
\end{align*}
which proves Theorem \ref{th4-1}. \qed
\end{proof}

The following theorem is an extension of Theorem \ref{th3-2}.


\begin{thm}
\label{th4-2}
Let $a:\mathbb{N} \times \mathbb{N} \mapsto \mathbb{C}$ be an arithmetic function of two variables. If the series 
\begin{equation}
\nonumber
A(q_1,q_2):=q_1 q_2 \sum_{k_1,k_2=1}^\infty  a(k_1 q_1,k_2 q_2) 
\end{equation}
converges for every $q_1,q_2 \in \mathbb{N}$, then for $f(q_1,q_2)=(A* \mu)(q_1,q_2)$, we have
\begin{equation}
\label{eq:th4-2}
f(q_1,q_2)=\sum_{n_1,n_2=1}^\infty a(n_1,n_2) c_{q_1} (n_1) c_{q_2} (n_2).
\end{equation}
\end{thm}
\begin{proof}
The proof proceeds along the same lines as the proof of Theorem \ref{th3-2}. We have
\begin{align*}
\sum_{\substack{n_1 \leqq x \\ n_2 \leqq y} } a(n_1,n_2) c_{q_1} (n_1) c_{q_2} (n_2)  = & \sum_{\substack{n_1 \leqq x \\ n_2 \leqq y} } a(n_1,n_2)  (\mu \underset{\ell}{*} D \underset{r}{*}  \bold{1} )(q_1,n_1) (\mu \underset{\ell}{*} D \underset{r}{*}  \bold{1} )(q_2,n_2) \\
= &  \sum_{\substack{n_1 \leqq x \\ n_2 \leqq y} } a(n_1,n_2)  (\sum_{d_1|n_1} (\mu  \underset{\ell}{*}  D)(q_1,d_1) \bold{1} (\frac{n_1}{d_1}) )(  \sum_{d_2|n_2} (\mu  \underset{\ell}{*}  D)(q_2,d_2)\bold{1} (\frac{n_2}{d_2}) ).
\end{align*}
Setting $n_1=d_1 k_1$, $n_2=d_2 k_2$ and using Lemma \ref{lem2-2}, we see that the above is equal to
\begin{align*}
& \sum_{\substack{d_1 k_1 \leqq x \\ d_2 k_2 \leqq y}} a(d_1 k_1, d_2 k_2) (\mu  \underset{\ell}{*}  D) (q_1,d_1)  (\mu  \underset{\ell}{*}  D) (q_2,d_2) \\
= &  \sum_{\substack{d_1 k_1 \leqq x \\ d_2 k_2 \leqq y}} a(d_1 k_1, d_2 k_2) I_{d_1|q_1} \mu(\frac{q_1}{d_1}) d_1  I_{d_2|q_2} \mu(\frac{q_2}{d_2}) d_2 \\
= &  \sum_{\substack{d_1  \leqq x \\ d_2 \leqq y}} I_{d_1|q_1} I_{d_2|q_2} \mu(\frac{q_1}{d_1}) \mu(\frac{q_2}{d_2}) (d_1 d_2 \sum_{\substack{k_1 \leqq x/d_1 \\ k_2 \leqq y/d_2} } a(d_1 k_1,d_2 k_2)),
\end{align*}
where $x,y$ are sufficiently large real numbers. Letting $x,y  \to   \infty$, we have 
\begin{align*}
\sum_{n_1,n_2=1}^\infty a(n_1,n_2) c_{q_1} (n_1) c_{q_2} (n_2)  =& \sum_{d_1,d_2=1}^\infty I_{d_1|q_1} I_{d_2|q_2} \mu(\frac{q_1}{d_1}) \mu(\frac{q_2}{d_2}) A(d_1,d_2) \\
  =& \sum_{\substack{d_1|q_1 \\ d_2 |q_2}}  \mu(\frac{q_1}{d_1}) \mu(\frac{q_2}{d_2})  A(d_1,d_2) =(\mu * A)(q_1,q_2) =f(q_1,q_2),
\end{align*}
which proves Theorem \ref{th4-2}. \qed
\end{proof}

The following example is an extension of Example \ref{ex3-2}.

\begin{exa}
\label{ex4-1}
Let $s>1.$ Then we have
\begin{equation}
\label{eq:ex4-1}
\Bigl(\prod_{p \in \mathcal{P}} (1-\frac{2}{p^s}) \Bigr) (\frac{q_1 q_2 \mu (q_1 q_2)}{ \widetilde{\varphi_s} (q_1)\widetilde{\varphi_s} (q_2)  } * \mu) (q_1,q_2)=\sum_{n_1,n_2=1}^\infty \frac{\mu(n_1 n_2)}{(n_1 n_2)^s} c_{q_1} (n_1) c_{q_2} (n_2),
\end{equation}
where $\widetilde{\varphi_s} (q)=q^s \prod_{p |q} (1-2/p^s).$ 
\end{exa}
\begin{proof}
Setting $a(n_1,n_2)=\frac{\mu(n_1 n_2)}{(n_1 n_2)^s}$ we have
\begin{align*}
A(q_1,q_2) & =q_1 q_2 \sum_{k_1,k_2=1}^\infty  a(k_1 q_1,k_2 q_2)=q_1 q_2 \sum_{k_1,k_2=1}^\infty \frac{\mu(k_1 q_1 k_2 q_2)}{(k_1 q_1 k_2 q_2)^s} \\
& = q_1 q_2 \sum_{\substack{k_1,k_2 \geqq 1 \\ (k_1 k_2, \ q_1 q_2)=1}}  \frac{\mu(k_1  k_2 ) \mu(q_1 q_2)}{(k_1  k_2 )^s (q_1  q_2)^{s}} = \frac{\mu(q_1 q_2)}{(q_1  q_2)^{s-1}} \sum_{\substack{k \geqq 1 \\ (k, \ q_1 q_2)=1}}  \frac{\mu(k )}{k^s} \sum_{k_1 \mid k} 1 \\
&= \frac{\mu(q_1 q_2)}{(q_1  q_2)^{s-1}} \sum_{\substack{k \geqq 1 \\ (k, \ q_1 q_2)=1}}  \frac{\mu(k ) \tau (k)}{k^s}  = \frac{\mu(q_1 q_2)}{(q_1  q_2)^{s-1}} \prod_{p \nmid q_1 q_2} (1+\frac{\mu(p) \tau(p)}{p^s}) \\
&= \frac{\mu(q_1 q_2)}{(q_1  q_2)^{s-1}} \prod_{p \nmid q_1 q_2} (1-\frac{2}{p^s}) =\frac{\mu(q_1 q_2)}{(q_1  q_2)^{s-1}} \frac{ \prod_{p \in \mathcal{P}} (1-2/p^s)}{\prod_{p \mid q_1 q_2} (1-2/p^s)} .
\end{align*}
If $(q_1,q_2)>1$, then $A(q_1,q_2) =0$ since $\mu(q_1 q_2)=0$. If $(q_1,q_2)=1$, then we have
\begin{align*}
A(q_1,q_2)  & = \frac{\mu(q_1 q_2)}{(q_1  q_2)^{s-1}} \frac{ \prod_{p \in \mathcal{P}} (1-2/p^s)}{\prod_{p \mid q_1 } (1-2/p^s) \prod_{p \mid q_2} (1-2/p^s)}   \\
 &= \frac{q_1 q_2 \mu(q_1 q_2)}{(q_1^s \prod_{p \mid q_1} (1-2/p^s))( q_2^s \prod_{p \mid q_2} (1-2/p^s))} \prod_{p \in \mathcal{P}} (1-2/p^s) \\
& = \frac{q_1 q_2 \mu (q_1 q_2)}{ \widetilde{\varphi_s} (q_1)\widetilde{\varphi_s} (q_2)  }   \prod_{p \in \mathcal{P}} (1-2/p^s),
\end{align*}
which clearly holds also in the case $(q_1,q_2)>1$.
If we set $f=A*\mu$, then Theorem \ref{th4-2} gives the desired result. \qed
\end{proof}


\begin{rem}
We consider the case $s \downarrow 1$ in (\ref{eq:ex4-1}), where the notation $s \downarrow 1$ means that $s$ approaches $1$ from above. Since
\begin{align*}
\prod_{p \in \mathcal{P}} (1-\frac{2}{p^s}) = &  (1-\frac{2}{2^s}) \prod_{\substack{p \in \mathcal{P} \\ p \geqq 3}} (1-\frac{2}{p^s}) = (1-\frac{2}{2^s}) \prod_{\substack{p \in \mathcal{P} \\ p \geqq 3}} (1-2/p^s) \frac{(1-1/p^s)^2}{1-2/p^s+1/p^{2s}} \\
 = & (1-\frac{2}{2^s}) \prod_{\substack{p \in \mathcal{P} \\ p \geqq 3}}  \frac{(1-1/p^s)^2}{\frac{1-2/p^s+1/p^{2s}}{(1-2/p^s)}} = (1-\frac{2}{2^s}) \prod_{\substack{p \in \mathcal{P} \\ p \geqq 3}} \frac{(1-1/p^s)^2}{1+\frac{1}{p^{2s} (1-2/p^s)}} \\
=& (1-\frac{2}{2^s}) \prod_{\substack{p \in \mathcal{P} \\ p \geqq 3}} (1-\frac{1}{p^s})^2 \prod_{\substack{p \in \mathcal{P} \\ p \geqq 3}} \frac{1}{1+\frac{1}{p^{s} (p^s-2)}}  \\
= & (1-\frac{2}{2^s}) \frac{1}{(1-\frac{1}{2^s})^2 \zeta^2 (s)} \prod_{\substack{p \in \mathcal{P} \\ p \geqq 3}} \frac{1}{1+\frac{1}{p^{s} (p^s-2)}} ,
\end{align*}
we have
\begin{align*}
& \lim_{s \downarrow  1} \Bigl(\prod_{p \in \mathcal{P}} (1-\frac{2}{p^s}) \Bigr) \frac{q_1 q_2 \mu (q_1 q_2)}{ \widetilde{\varphi_s} (q_1)\widetilde{\varphi_s} (q_2)  } \\
  = &\lim_{s \downarrow  1}   (1-\frac{2}{2^s})  \frac{1}{(1-\frac{1}{2^s})^2 \zeta^2 (s)} \Bigl(\prod_{\substack{p \in \mathcal{P} \\ p \geqq 3}} \frac{1}{1+\frac{1}{p^s (p^s-2)}} \Bigr)  \frac{q_1 q_2 \mu (q_1 q_2)}{ \widetilde{\varphi_s} (q_1)\widetilde{\varphi_s} (q_2)  }  \\
 =&  \lim_{s \downarrow  1}   (1-\frac{2}{2^s})  \frac{1}{\zeta^2 (s)} \frac{q_1 q_2 \mu (q_1 q_2)}{q_1^s \prod_{p |q_1} (1-2/p^s) q_2^s \prod_{p |q_2} (1-2/p^s)    }  \frac{1}{(1-\frac{1}{2})^2} \prod_{\substack{p \in \mathcal{P} \\ p \geqq 3}} \frac{1}{1+\frac{1}{p (p-2)}} \\
= & \lim_{s \downarrow  1}  \frac{1}{\zeta^2 (s) } \frac{ 1-2/2^s }{ \prod_{p |q_1} (1-2/p^s)  \prod_{p |q_2} (1-2/p^s)} \mu (q_1 q_2) \frac{1}{(1-\frac{1}{2})^2} \prod_{\substack{p \in \mathcal{P} \\ p \geqq 3}} \frac{1}{1+\frac{1}{p (p-2)}}=0,
\end{align*}
where we note that, since $\mu(q_1 q_2)=0$ if $q_1$ and $q_2$ are even, we may assume $q_1$ or $q_2$ is odd. 
Therefore by letting $s \downarrow 1$ in (\ref{eq:ex4-1}), we obtain 
\begin{equation}
\nonumber
\sum_{n_1,n_2=1}^\infty \frac{\mu(n_1 n_2)}{n_1 n_2} c_{q_1} (n_1) c_{q_2} (n_2)=0,
\end{equation}
which is an extension of (\ref{ram3-3}) to the case of two variables.
Of course an extension of (\ref{eq:ex3-3}) 
\begin{equation}
\nonumber
\sum_{n_1,n_2=1}^\infty \frac{\lambda (n_1 n_2 )}{n_1 n_2} c_{q_1} (n_1) c_{q_2} (n_2)=0
\end{equation}
clearly holds since $\lambda$ is completely multiplicative.
\end{rem}

Next we consider extensions of Theorem \ref{th:Delange} and Theorem \ref{th3-4}.
Ushiroya \cite{Ushiroya3} proved the following theorem which is an extension of Theorem \ref{th:Delange}. 


\begin{thm}[\cite{Ushiroya3}]
{\rm{(i)}}
Let $f(n_1,n_2)$ be an arithmetic function of two variables satisfying
\begin{equation}
\nonumber
\sum_{n_1,n_2=1}^\infty 2^{\omega(n_1)} 2^{\omega(n_2)} \frac{|(f*\mu) (n_1,n_2)|}{n_1 n_2} < \infty . 
\end{equation}
Then its Ramanujan-Fourier series is pointwise convergent and
\begin{equation}
\nonumber
f(n_1,n_2)=\sum_{q_1,q_2=1}^\infty a(q_1,q_2) c_{q_1} (n_1)c_{q_2} (n_2) 
\end{equation}
holds where 
\begin{equation}
\nonumber
a(q_1,q_2)  = \sum_{m_1,m_2=1}^\infty \frac{(f * \mu) (m_1 q_1, m_2 q_2) }{m_1 q_1 m_2 q_2}.  
\end{equation}

{\rm{(ii)}}
Let $f$ be a multiplicative function of two variables satisfying
\begin{equation}
\nonumber
 \sum_{p \in \mathcal{P} } \sum_{\substack{e_1 ,e_2  \geqq 0 \\ e_1+e_2 \geqq 1}} \frac{|(f * \mu ) (p^{e_1}, p^{e_2})|}{p^{e_1+e_2}}<\infty.
\end{equation}
Then its Ramanujan-Fourier series is pointwise convergent and
\begin{equation}
\nonumber
f(n_1,n_2)=\sum_{q_1,q_2=1}^\infty a(q_1,q_2) c_{q_1} (n_1)c_{q_2} (n_2) 
\end{equation}
holds where 
\begin{equation}
\nonumber
a(q_1,q_2)  = \sum_{m_1,m_2=1}^\infty \frac{(f * \mu) (m_1 q_1, m_2 q_2) }{m_1 q_1 m_2 q_2} .
\end{equation}
Moreover, if the mean value $M(f)=\lim_{x \to \infty} \sum_{n \leqq x} f(n) $ is not zero and if $\{ q_1,q_2 \} >1$, where $\{q_1,q_2 \}$ denotes the least common multiple of $q_1$ and $q_2$
, then $a(q_1,q_2) $ can be rewritten as
\begin{align}
a(q_1,q_2) & =\prod_{p \in \mathcal{P}} \Bigl( \sum_{e_1=\nu_p (q_1)} \sum_{e_2=\nu_p (q_2)} \frac{(f * \mu ) (p^{e_1}, p^{e_2})}{p^{e_1+e_2}} \Bigr) \nonumber \\
\label{eq:th2-4-2}
 & =M(f) \prod_{p \mid \{q_1,q_2 \}} \Bigl\{ \Bigl( \sum_{e_1=\nu_p (q_1)} \sum_{e_2=\nu_p (q_2)} \frac{(f * \mu ) (p^{e_1}, p^{e_2})}{p^{e_1+e_2}} \Bigr) / \Bigl( \sum_{e_1=0} \sum_{e_2=0} \frac{(f * \mu ) (p^{e_1}, p^{e_2})}{p^{e_1+e_2}} \Bigr) \Bigr\}.
\end{align}
\end{thm}

We remark that many examples of the form 
\begin{equation}
\nonumber
f(n_1,n_2)=\sum_{q_1,q_2=1}^\infty a(q_1,q_2) c_{q_1} (n_1) c_{q_2} (n_2)
\end{equation}
are obtained in \cite{Ushiroya3}.

Next we extend Theorem \ref{th3-4} to dual Ramanujan-Fourier series.


\begin{thm}
\label{th4-4}
Let $f$ be an arithmetic function of two variables satisfying
\begin{equation}
\label{eq1:th4-4}
\sum_{q_1,q_2=1}^\infty  \frac{|(f * \bold{1})(q_1,q_2)|}{q_1 q_2} \tau(q_1) \tau(q_2) < \infty.
\end{equation}
Then its dual Ramanujan-Fourier series is pointwise convergent and
\begin{equation}
\nonumber
f(q_1,q_2)=\sum_{n_1,n_2=1}^\infty a(n_1,n_2) c_{q_1} (n_1) c_{q_2} (n_2)
\end{equation}
holds where 
\begin{equation}
\label{eq2:th4-4}
a(n_1,n_2)=\sum_{m_1,m_2=1}^\infty \frac{(f * \bold{1})(n_1 m_1,n_2m_2)}{n_1 m_1 n_2 m_2} \mu(m_1,m_2).  
\end{equation}
\end{thm}


\begin{rem} 
\label{rem4-2}
Let $f$ be a multiplicative function of two variables satisfying
\begin{equation}
\label{eq3:th4-4}
\sum_{p \in \mathcal{P}} \sum_{\substack{e_1,e_2 \geqq 0 \\ e_1+e_2 \geqq 1}} \frac{|(f* \bold{1}) (p^{e_1},p^{e_2})|}{p^{e_1+e_2} } (e_1+1)(e_2+1) < \infty.
\end{equation}
Then its dual Ramanujan-Fourier series is pointwise convergent and
\begin{equation}
\nonumber
f(q_1,q_2)=\sum_{n_1,n_2=1}^\infty a(n_1,n_2) c_{q_1} (n_1) c_{q_2} (n_2)
\end{equation}
holds where 
\begin{equation}
\nonumber
a(n_1,n_2)=\sum_{m_1,m_2=1}^\infty \frac{(f * \bold{1})(n_1 m_1,n_2 m_2)}{n_1 m_1 n_2 m_2} \mu(m_1,m_2).  
\end{equation}
Moreover, $a(n_1,n_2)$ can be rewritten as
\begin{align*}
a(n_1,n_2)=\prod_{p \in \mathcal{P}} \Bigl( &  \frac{(f * \bold{1})(p^{\nu_p(n_1)},p^{\nu_p(n_2)})}{p^{\nu_p(n_1) + \nu_p(n_2)}} -\frac{(f * \bold{1})(p^{\nu_p(n_1)+1},p^{\nu_p(n_2)})}{p^{\nu_p(n_1) + \nu_p(n_2)+1}}  \\
& -\frac{(f * \bold{1})(p^{\nu_p(n_1)},p^{\nu_p(n_2)+1})}{p^{\nu_p(n_1) + \nu_p(n_2)+1}}+\frac{(f * \bold{1})(p^{\nu_p(n_1)+1},p^{\nu_p(n_1)+2})}{p^{\nu_p(n_1)+\nu_p(n_2)+2}} \Bigr).
\end{align*}

\end{rem}

\begin{proof}[Proof of Theorem \ref{th4-4}.]
We proceed along the same lines as the proof of Theorem \ref{th3-4}.
We first see that $A(q_1,q_2)=q_1 q_2 \sum_{k_1,k_2=1}^\infty a(k_1 q_1,k_2 q_2)$ converges since
\begin{align*}
\sum_{\substack{k_1 \leqq x \\ k_2 \leqq y}} |a(k_1 q_1,k_2 q_2)| & \leqq \sum_{\substack{k_1 \leqq x \\ k_2 \leqq y}} \sum_{m_1,m_2=1}^\infty \frac{|(f * \bold{1})(k_1 q_1 m_1,k_2 q_2 m_2)|}{k_1 q_1 m_1 k_2 q_2 m_2} |\mu(m_1,m_2)| \\
 & \leqq \sum_{\ell_1,\ell_2=1}^\infty  \frac{|(f* \bold{1}) (\ell_1,\ell_2)|}{\ell_1 \ell_2} \sum_{\substack{k_1 \leqq x, \ k_1 | \ell_1 \\ k_2 \leqq y, \ k_2 | \ell _2}} 1 \\
 & \leqq \sum_{\ell_1,\ell_2=1}^\infty  \frac{|(f* \bold{1}) (\ell_1,\ell_2)|}{\ell_1 \ell_2} \tau(\ell_1) \tau(\ell_2) < \infty.
\end{align*}
Using Lemma \ref{lem2-5} we can rewrite $A(q_1,q_2)$ as
\begin{align*}
A(q_1,q_2) &=q_1 q_2 \sum_{k_1,k_2=1}^\infty a(k_1 q_1,k_2 q_2)= \sum_{m_1,m_2=1}^\infty a(m_1,m_2) q_1 I_{q_1|m_1} q_2 I_{q_2 |m_2}  \\
& = \sum_{m_1,m_2=1}^\infty a(m_1,m_2) ( \bold{1} \underset{\ell}{*} c)(q_1,m_1) ( \bold{1} \underset{\ell}{*} c)(q_2,m_2).
\end{align*}
From this we have
\begin{align*}
f(q_1,q_2) &=(\mu * A)(q_1,q_2)=\sum_{m_1,m_2=1}^\infty a(m_1,m_2) (\mu \underset{\ell}{*} ( \bold{1} \underset{\ell}{*} c))(q_1,m_1) (\mu \underset{\ell}{*} ( \bold{1} \underset{\ell}{*} c))(q_2,m_2) \\
& =\sum_{m_1,m_2=1}^\infty a(m_1,m_2) (\delta \underset{\ell}{*} c)(q_1,m_1) (\delta \underset{\ell}{*} c)(q_2,m_2) \\
 &=\sum_{m_1,m_2=1}^\infty a(m_1,m_2) c_{q_1} (m_1) c_{q_2} (m_2).
\end{align*}
This completes the proof of Theorem \ref{th4-4}. \qed
\end{proof}

\begin{proof}[Proof of Remark \ref{rem4-2}.]
We first note that, if $f$ is a multiplicative function of two variables, then $(q_1,q_2) \mapsto \frac{(f * \bold{1})(q_1,q_2)}{q_1 q_2} \tau(q_1) \tau(q_2)$ is also a multiplicative function of two variables. Using $1+x \leqq \exp(x)$, we see that (\ref{eq1:th4-4}) follows from (\ref{eq3:th4-4}) since 
\begin{align*}
\sum_{\substack{q_1 \leqq Q_1 \\ q_2 \leqq Q_2}}  \frac{|(f * \bold{1})(q_1,q_2)|}{q_1 q_2} \tau(q_1) \tau(q_2) \leqq & \prod_{p \in \mathcal{P}} \Bigl(1+\sum_{\substack{e_1,e_2 \geqq 0 \\ e_1+e_2 \geqq 1}} \frac{|(f * \bold{1})(p^{e_1},p^{e_2})|}{p^{e_1} p^{e_2} } \tau(p^{e_1}) \tau(p^{e_2}) \Bigr) \\
 \leqq & \prod_{p \in \mathcal{P}} \exp \Bigl( \sum_{\substack{e_1,e_2 \geqq 0 \\ e_1+e_2 \geqq 1}} \frac{|(f * \bold{1})(p^{e_1},p^{e_2})|}{p^{e_1} p^{e_2} } \tau(p^{e_1}) \tau(p^{e_2}) \Bigr) \\
 = & \exp \Bigl( \sum_{p \in \mathcal{P}} \sum_{\substack{e_1,e_2 \geqq 0 \\ e_1+e_2 \geqq 1}} \frac{|(f * \bold{1})(p^{e_1},p^{e_2})|}{p^{e_1+e_2} }  (e_1+1)(e_2+1) \Bigr) < \infty
\end{align*}
holds for any $Q_1, Q_2>1$. Therefore (\ref{eq2:th4-4}) holds by Theorem \ref{th4-4}. In the expression of (\ref{eq2:th4-4}),  we set, for $i=1,2$, $n_i=\prod p_j^{e_{ij}} $,  $m_i=r_i \prod p_j^{d_{ij}}$  where  $(r_i,n_1 n_2)=1$,  $e_{ij} \geqq 1$, and $d_{ij} \geqq 0$.
Then we have
\begin{align*}
a(n_1,n_2)= & \sum_{\substack{d_{ij} \geqq 0, \ r_i \geqq 1 \\ (r_i,n_1 n_2)=1}} \frac{(f*\bold{1}) (r_1 \prod p_j^{d_{1j}+e_{1j}}, r_2 \prod p_j^{d_{2j}+e_{2j}})}{r_1 r_2 \prod p_j^{d_{1j}+e_{1j}+d_{2j}+e_{2j}}} \mu(r_1 \prod p_j^{d_{1j}}, r_2 \prod p_j^{d_{2j}}) .
\end{align*}
Since $f*\bold{1}$ is multiplicative and since $\mu( p_j^{d_{1j}}, p_j^{d_{2j}}) = 0$ if $d_{1j} \geqq 2$ or  $d_{2j} \geqq 2$ for some $j$, we obtain
\begin{align*}
a(n_1,n_2)=  & \prod_j \Bigl( \sum_{0 \leqq d_{ij} \leqq 1} \frac{(f*\bold{1}) (  p_j^{d_{1j}+e_{1j}}, p_j^{d_{2j}+e_{2j}})}{ p_j^{d_{1j}+e_{1j}+d_{2j}+e_{2j}}} \mu(  p_j^{d_{1j}},  p_j^{d_{2j}}) \Bigr) \times \sum_{\substack{r_i \geqq 1 \\ (r_,n_1 n_2)=1}} \frac{(f*\bold{1}) (r_1,r_2 )}{r_1 r_2 } \mu(r_1,r_2) \\
= & \prod_{p|n_1 n_2}  \Bigl(  \frac{(f * \bold{1})(p^{\nu_p(n_1)},p^{\nu_p(n_2)})}{p^{\nu_p(n_1) + \nu_p(n_2)}} -\frac{(f * \bold{1})(p^{\nu_p(n_1)+1},p^{\nu_p(n_2)})}{p^{\nu_p(n_1) + \nu_p(n_2)+1}}  \\
  & \quad \quad \quad - \frac{(f * \bold{1})(p^{\nu_p(n_1)},p^{\nu_p(n_2)+1})}{p^{\nu_p(n_1) + \nu_p(n_2)+1}}+\frac{(f * \bold{1})(p^{\nu_p(n_1)+1},p^{\nu_p(n_1)+2})}{p^{\nu_p(n_1)+\nu_p(n_2)+2}} \Bigr) \\
& \quad \quad \quad \times \prod_{p \nmid n_1 n_2}  (1-\frac{f(p,1)}{p}-\frac{f(1,p)}{p}+\frac{f(p,p)}{p^2}) \\
= & \prod_{p \in \mathcal{P}} \Bigl(   \frac{(f * \bold{1})(p^{\nu_p(n_1)},p^{\nu_p(n_2)})}{p^{\nu_p(n_1) + \nu_p(n_2)}} -\frac{(f * \bold{1})(p^{\nu_p(n_1)+1},p^{\nu_p(n_2)})}{p^{\nu_p(n_1) + \nu_p(n_2)+1}}  \\
& \quad  \quad -\frac{(f * \bold{1})(p^{\nu_p(n_1)},p^{\nu_p(n_2)+1})}{p^{\nu_p(n_1) + \nu_p(n_2)+1}}+\frac{(f * \bold{1})(p^{\nu_p(n_1)+1},p^{\nu_p(n_1)+2})}{p^{\nu_p(n_1)+\nu_p(n_2)+2}} \Bigr),
\end{align*}
which completes the proof of Remark \ref{rem4-2}. \qed
\end{proof}

If we take $f=\mu$ in Theorem \ref{th3-4}, then it is obvious that 
\begin{equation}
\nonumber
\mu(q)=\sum_{n=1}^\infty a(n) c_q(n)
\end{equation}
holds where $a(n)=\delta(n).$ The following example is an extension of the above trivial example.


\begin{exa}
\label{ex4-2}
We have
\begin{equation}
\nonumber
 \mu (q_1 q_2)= \frac{1}{\zeta(2)} \sum_{n_1,n_2=1}^\infty \frac{\mu (K((n_1,n_2))) \varphi  (K((n_1,n_2)))}{n_1 n_2 \psi (K( n_1 n_2 ))}  c_{q_1} (n_1) c_{q_2} (n_2) .
\end{equation}
\end{exa}

\begin{proof}
Let $f(q_1,q_2)=\mu (q_1 q_2).$ Then it is easy to see that
\begin{align*}
& (f*\bold{1})(p^k,1)=(f*\bold{1})(1,p^k)=0 \ \ \mathrm{if} \ \ k \geqq 1, \\
& (f*\bold{1})(p^k,p^{\ell})=-1   \quad \mathrm{if} \ \ k, \ell \geqq 1. 
\end{align*}
From this we see that (\ref{eq3:th4-4}) holds. We have
\begin{align*}
& \frac{(f*\bold{1})(p^k,p^{\ell})}{p^{k+\ell}}-\frac{(f*\bold{1})(p^{k+1},p^{\ell})}{p^{k+\ell+1}}-\frac{(f*\bold{1})(p^k,p^{\ell+1})}{p^{k+\ell+1}}+\frac{(f*\bold{1})(p^{k+1},p^{\ell+1})}{p^{k+\ell+2}} \\
& = \left\{ \begin{array}{ll} 
  1-1/p^2  & \mathrm{if} \ \ k=\ell=0  \\ 1/p^{\ell+1} -1/p^{\ell+2} &  \mathrm{if}  \ \ k=0, \ \ \ell \geqq 1 \\ 1/p^{k+1} -1/p^{k+2} &  \mathrm{if}  \ \ k \geqq 1, \ \ \ell =0 \\ -1/p^{k+\ell} + 2/p^{k+\ell+1} -1/p^{k+\ell+2} &  \mathrm{if}  \ \ k , \ell \geqq 1 . \end{array}  \right.
\end{align*}
Therefore we have by Remark \ref{rem4-2} 
\begin{align*}
a(n_1,n_2) & = \prod_{p \nmid n_1 n_2 } (1-\frac{1}{p^2}) \prod_{ \substack{p \nmid n_1 \\ p \mid n_2} } ( \frac{1}{p^{\nu_p (n_2)+1}} -  \frac{1}{p^{\nu_p (n_2)+2}}) \prod_{ \substack{p \mid n_1 \\ p \nmid n_2} } ( \frac{1}{p^{\nu_p (n_1)+1}} -  \frac{1}{p^{\nu_p (n_1)+2}}) \\
& \quad \quad \times \prod_{ \substack{p \mid n_1 \\ p \mid n_2} }  (-\frac{1}{p^{\nu_p (n_1)+\nu_p (n_2)}} + \frac{2}{p^{\nu_p (n_1)+\nu_p (n_2)+1}}  - \frac{1}{p^{\nu_p (n_1)+\nu_p (n_2)+2}} )  \\
& = \prod_{p \nmid n_1 n_2 } (1-\frac{1}{p^2}) \prod_{ \substack{p \nmid n_1 \\ p \mid n_2} }  \frac{1}{p^{\nu_p (n_2)+1}} ( 1 -  \frac{1}{p}) \prod_{ \substack{p \mid n_1 \\ p \nmid n_2} }  \frac{1}{p^{\nu_p (n_1)+1}} ( 1 -  \frac{1}{p}) \\
& \quad  \quad \times \prod_{ \substack{p \mid n_1 \\ p \mid n_2} }  \frac{-1}{p^{\nu_p (n_1)+\nu_p (n_2)}}  (1-\frac{2}{p}+\frac{1}{p^2}) \\
 & = \prod_{p \in \mathcal{P}} (1-\frac{1}{p^2}) \prod_{ \substack{p \nmid n_1 \\ p \mid n_2} }  \frac{1}{p^{\nu_p (n_2)+1}} \frac{1-1/p}{1-1/p^2} \prod_{ \substack{p \mid n_1 \\ p \nmid n_2} }  \frac{1}{p^{\nu_p (n_1)+1}} \frac{1-1/p}{1-1/p^2}  \prod_{ \substack{p \mid n_1 \\ p \mid n_2} }  \frac{-1}{p^{\nu_p (n_1)+\nu_p (n_2)}}  \frac{(1-1/p)^2}{1-1/p^2} \\
& = \frac{1}{\zeta(2)} \prod_{ \substack{p \nmid n_1 \\ p \mid n_2} }  \frac{1}{p^{\nu_p (n_2)+1}} \frac{1}{1+1/p} \prod_{ \substack{p \mid n_1 \\ p \nmid n_2} }  \frac{1}{p^{\nu_p (n_1)+1}} \frac{1}{1+1/p}    \prod_{ p \mid (n_1 , n_2) }  \frac{-1}{p^{\nu_p (n_1)+\nu_p (n_2)}}  \frac{1-1/p}{1+1/p}  \\  
& = \frac{1}{\zeta(2)} \prod_{p \mid n_1 n_2 } \frac{1}{p^{\nu_p(n_1)+\nu_p (n_2)}}  \frac{1}{p(1+1/p)}  \prod_{p \mid (n_1, n_2) }  (-1)p(1-1/p)   \\
& = \frac{1}{\zeta(2)}  \frac{\mu (K((n_1,n_2))) \varphi  (K((n_1,n_2)))}{n_1 n_2 \psi (K( n_1 n_2 ))} ,
\end{align*}
which completes the proof of Example \ref{ex4-2}. \qed
\end{proof}

\hspace{-0.7cm} {\bf{Acknowledgment. }} \\
The author sincerely thanks the referees, whose comments and suggestions essentially improved this paper.


\bibliographystyle{amsalpha}

\authoraddresses{
Noboru Ushiroya\\
National Institute of Technology, Wakayama College, \\
77 Noshima Nada Gobo Wakayama, Japan \\
\email ushiroya@wakayama-nct.ac.jp

}

\end{document}